\documentclass[11pt]{article}

\usepackage{amsfonts}
\usepackage{amssymb}
\usepackage{amsmath}
\usepackage{amsthm}
\usepackage{graphicx}
\usepackage{color}
\usepackage[english]{babel}

\newtheorem{notation}{Notations}[section]
\newtheorem{remarque}[notation]{Remark}
\newtheorem{thm}[notation]{Theorem}

\newtheorem{defin}[notation]{Definition} 

\newtheorem{prop}[notation]{Proposition}
\newtheorem{lem}[notation]{Lemma}

\newcommand{\gga}{\gamma}           
\newcommand{\gd}{\delta}
\newcommand{\gphi}{\varphi}

\newcommand{\gG}{\Gamma}

\newcommand{\gD}{\Delta}

\newcommand{\gl}{\lambda}

\newcommand{\eps}{\varepsilon}

\newcommand{\R}{\mathbb R}

\newcommand{\N}{\mathbb{N}}
\newcommand{\E}{\mathbb E}

\newcommand{\cC}{\mathcal C}
\newcommand{\cH}{\mathcal{H}}
\newcommand{\cU}{\mathcal U}
\newcommand{\cM}{\mathcal M}
\newcommand{\bbar}{\big{|}}

\newcommand{\dive}{\operatorname{div}}

\newcommand{\bEm}{\mathbb{E}_m}
\newcommand{\FCb}{\mathcal{F}\cC_b^1}
\newcommand{\FCc}{\mathcal{F}\cC_c^1}
\newcommand{\barf}{\overline{F}}
\newcommand{\hath}{\hat{h}}
\newcommand{\vgamma}[1]{\gga\left(#1\right)}
\newcommand{\Pgamma}[1]{P_\gga(#1)}

\newcommand{\normH}[1]{|#1|_H}

\newcommand{\totvar}[1]{\int_X \normH{D_\gga #1} }
\newcommand{\totvarn}[1]{\int_{\R^m} |D_{\gga_m} #1| }
\newcommand{\sprod}[2]{\langle #1, #2 \rangle}

\newcommand{\Ldeu}{L^2_\gga(X)}

\begin{document}
\title{\Large Approximation and relaxation of perimeter in the Wiener space}

\author{
        M. Goldman
        \footnote{CMAP, CNRS UMR 7641, Ecole Polytechnique,
        91128 Palaiseau, France, email: goldman@cmap.polytechnique.fr}
        \and M. Novaga
        \footnote{Dipartimento di Matematica, Universit\`a di Padova,
        via Trieste 63, 35121 Padova, Italy, email: novaga@math.unipd.it}
}
\date{}
\maketitle

\begin{abstract}
\noindent We characterize the relaxation of the perimeter in an infinite dimensional Wiener space, with respect to the weak $L^2$-topology.
We also show that the rescaled Allen-Cahn functionals approximate this relaxed functional in the sense of $\Gamma$-convergence.
\end{abstract}

\section{Introduction}

\noindent Extending the variational methods and the geometric measure theory from the Euclidean 
to the Wiener space  
has recently attracted a lot of attention.
In particular, the theory of functions of bounded variation in infinite dimensional spaces started with the works by Fukushima and Hino 
\cite{fuku,fukuhino}. Since then, the fine properties of $BV$ functions and sets of finite perimeter have been investigated 
in \cite{AMMP,AMP,AF,ADPP}. 
We point out that this theory is closely related to older works by M. Ledoux and P. Malliavin \cite{ledoux,malliavin}.  

\noindent In the Euclidean setting it is well-known that the perimeter 
can be approximated by means of more regular functionals of the form 
\[
\int \left(\frac{\eps}{2}|\nabla u|^2 + \frac{W(u)}{\eps}\right)\, dx
\]
when $\eps$ tends to zero, in the sense of $\Gamma$-convergence with respect to the strong $L^1$-topology 
\cite{modicamortola,modica}. An important ingradient in this proof is the compact embedding of $BV$ in $L^1$.

\noindent A natural question is whether a similar approximation property holds
in the infinite dimensional case.
The main goal of this paper is answering to this question 
by computing the $\gG$-limit, as $\eps\to 0$, of the Allen-Cahn-type functionals
(see Section \ref{notation} for precise definitions)
\[
F_\eps(u)= \int_X  \left(\frac{\eps}{2}\normH{\nabla_H u}^2 +\frac{W(u)}{\eps}\right) d\gga .
\]

\noindent In the Wiener space there are two possible definitions of gradient, and consequently two different notions of Sobolev spaces, functions of bounded variation and perimeters \cite{AMMP,ADPP}. In one definition the compact embedding of $BV_\gga(X)$ in $L^1_\gga(X)$ still holds \cite[Th. 5.3]{AMMP}
and the $\gG$-limit of $F_\eps$ is, as expected, the perimeter up to a multiplicative constant. 
We do not reproduce here the proof of this fact, since it is very similar to the Euclidean one. 

\noindent A more interesting situation arises when we consider the other definition of gradient, which gives rise 
to a more invariant notion of perimeter and is therefore commonly used in the literature \cite{fuku,fukuhino,AMMP}. 
In this case, the compact embedding of $BV_\gga(X)$ in $L^1_\gga(X)$ does not hold anymore. In particular 
sequences with uniformly bounded $F_\eps$-energy are not generally 
compact in the (strong) $L^1_\gga$-topology, even though they are bounded in $\Ldeu$, and hence compact with respect to the weak $\Ldeu$-topology.
This suggests that the right topology for considering the $\gG$-convergence should rather be the weak $\Ldeu$-topology. 

\noindent A major difference with the finite dimensional case is the fact that the perimeter function defined by
\[
F(u)=\left\{\begin{array}{ll}
\Pgamma{E} \qquad& \textrm{if } u=\chi_E
\\[8pt]
+\infty \qquad & \textrm{otherwise}
\end{array}\right.
\]
is no longer lower semicontinuous in this topology, and therefore cannot be the $\gG$-limit of the functionals $F_\eps$.
The problem is that the sets of finite perimeter are not closed under weak convergence of the characteristic functions. 
However, it is possible to compute the relaxation $\barf$ of $F$ (Theorem \ref{relax}), which reads:
\[
\barf(u)=\left\{\begin{array}{ll}
\displaystyle \int_X \sqrt{\cU^2(u)+|D_\gga u|^2}\,d \gga \qquad& \textrm{if } 0\le u\le 1 
\\[8pt]
+\infty \qquad & \textrm{otherwise.}
\end{array}\right.
\] 
Such functional is quite familiar to people studying log--Sobolev and isoperimetric inequalities in Wiener spaces 
\cite{bakryledoux,bobkov,CK}. 

\noindent Our main result is to show that the $\gG$-limit of $F_\eps$, with respect to the weak $\Ldeu$-topology,
is a multiple of $\barf$ (Theorem \ref{modmortMal}). 
The proof relies on the interplay between symmetrization, semicontinuity and isoperimetry.

\noindent The plan of the paper is the following. In Section \ref{notation} we recall some basic facts about Wiener spaces and functions of bounded variation. In Section \ref{rapehr} we give the main properties of the Ehrhard symmetrizations. We also prove a P\'olya-Szeg\"o inequality and a Bernstein-type result 
in the Wiener space (Propositions \ref{ehrfunc} and \ref{probern}), which we believe to be interesting in themselves.
In Section \ref{secrelax}, we use the Ehrhard symmetrization to compute the relaxation of the perimeter (Theorem \ref{relax}). 
Finally, in Section \ref{MMM} we compute the $\gG$-limit of the functionals $F_\eps$ (Theorem \ref{modmortMal}) and discuss some consequences of this result.

\smallskip

\noindent{\bf Acknowledgements:} The authors wish to thank Michele Miranda for valuable discussions. 
The first author would like also to thank the Scuola Normale di Pisa for the
kind hospitality, and Luigi Ambrosio for the invitation and the interest in this work.


\section{Wiener space and functions of bounded variation}\label{notation}

\noindent A clear and comprehensive reference on the Wiener space is the book by Bogachev \cite{boga} (see also \cite{malliavin}). 
We follow here closely the notation of \cite{AMMP}. Let $X$ be a separable Banach space and let $X^*$ be its dual. We say that $X$ is a Wiener space if it is endowed with a non-degenerate centered Gaussian probability measure $\gga$. That amounts to say that $\gga$ is a probability measure for which $x^*\sharp \gga$ is a centered Gaussian measure on $\R$ for every $x^*\in X^*$. The non-degeneracy hypothesis means that $\gga$ is not concentrated 
on any proper subspace of $X$.

\noindent As a consequence of Fernique's Theorem \cite[Th. 2.8.5]{boga}, for every $x^*\in X^*$, the function $R^* x^*(x)=\sprod{x^*}{x}$ is in $\Ldeu=L^2(X,\gga)$. Let $\cH$ be the closure of $R^*X^*$ in $\Ldeu$; the space $\cH$ is usually called the reproducing kernel of $\gga$. Let $R$, the operator from $\cH$ to $X$, be the adjoint of $R^*$ that is, for $\hath \in \cH$,
\[R \hath=\int_X x \hath(x)\, d\gga \]
where the integral is to be intended in the Bochner sense. It can be seen that $R$ is a compact and injective operator. We will let $Q=RR^*$. We denote by $H$ the space $R\cH$. This space is called the Cameron-Martin space. It is a separable Hilbert space with the scalar product given by
\[
[h_1,h_2]_H=\sprod{\hath_1}{\hath_2}_{\Ldeu}
\]
if $h_i=R\hath_i$. We will denote by $\normH{\cdot}$ the norm in $H$. 
The space $H$ is a dense subspace of $X$, with compact embedding, and $\gga(H)=0$ if $X$ is of infinite dimension. 

\noindent For $x_1^*, .., x_m^*\in X^*$ we denote by $\Pi_{x_1^*, .., x_m^*}$ the projection from $X$ to $\R^m$ given by
\[\Pi_{x_1^*,.., x_m^*}(x)=(\sprod{x_1^*}{x},.., \sprod{x_m^*}{x}).\]
We will also denote it by $\Pi_m$ when  specifying the points $x_i^*$ is unnecessary. Two elements $x_1^*$ and $x_2^*$ of $X^*$ will be called orthonormal if the corresponding $h_i=Qx_i^*$ are orthonormal in $H$. We will fix in the following an orthonormal base of $H$ given by $h_i=Q x_i^*$.

\noindent We also denote by $H_m=\textrm{span}(h_1, .., h_m)\simeq \R^m$ and $X_m^\perp=\textrm{Ker}(\Pi_m)=\overline{H_m^\perp}^X$,
so that $X=\R^m\oplus X_m^\perp$.
The map $\Pi_m$ induces the decomposition $\gga=\gga_m \otimes\gga_m^\perp$, with $\gga_m,\,\gga_m^\perp$ Gaussian measures
on $\R^m,\,X_m^\perp$ respectively.

\begin{prop}[\cite{boga}]
Let $\hath_1, .., \hath_m$ be in $\cH$ then the image measure of $\gga$ under the map
\[\Pi_{\hath_1, .., \hath_m} (x)= (\hath_1(x), .. , \hath_m(x))\]
is a Gaussian in $\R^m$. If the $\hath_i$ are orthonormal, then such measure is the standard Gaussian measure on $\R^m$.
\end{prop}

\noindent Given $u\in \Ldeu$, we will consider the canonical cylindrical approximation $\bEm$ given by
\[\bEm u (x)=\int_{X_m^\perp} u(\Pi_m(x), y) \,d\gga_m^\perp(y).\]
Notice that $\bEm u$ is a cylindrical functions depending only on the first $m$ variables, and $\bEm u$ converges  to $u$ in $\Ldeu$.

\noindent We will denote by $\FCb(X)$ the space of cylindrical  $\cC^1$ bounded functions that is the functions of the form $v(\Pi_m (x))$ with $v$ a $\cC^1$ bounded function from $\R^m$ to $\R$. 
We denote by $\FCb(X,H)$ the space generated by all functions of the form 
$\Phi h$, with $\Phi\in \FCb(X)$ and $h\in H$.

\noindent We now give the definitions of gradients, Sobolev spaces functions of bounded variation. 
Given $u: X\rightarrow \R$ and $h=R\hath \in H$, we define
\[\frac{\partial u} {\partial h} (x)=\lim_{t\to 0} \,\frac{u(x+th)-u(x)}{t}\]
whenever the limit exists, and
\[\partial_h^* u= \frac{\partial u}{\partial h} -\hath(u). \]
We define $\nabla_H u: X\rightarrow H$,  the gradient of $u$ by
\[\nabla_H u= \sum_{i=1}^{+\infty} \frac{\partial u}{\partial h_i}\,  h_i\]
and the divergence of $\Phi: X\rightarrow H$ by
\[\dive_\gga \Phi =\sum_{i=1}^{+\infty} \partial_{h_i}^* [\Phi,h_i]_H.\]

\noindent The operator $\dive_\gga$ is the adjoint of the gradient so that for every $u \in \FCb(X)$ and every $\Phi \in \FCb(X,H)$, the following integration by parts holds:
\begin{equation}\label{integpart}
\int_X u \dive_\gga \Phi \,d\gga =-\int_X [\nabla_H u, \Phi]_H d\gga.
\end{equation}
\noindent The $\nabla_H$ operator is thus closable in $\Ldeu$ and we will denote by $H^1_\gga(X)$ its closure in $\Ldeu$. From this, formula \eqref{integpart} still holds for $u \in H^1_\gga(X)$ and $\Phi \in \FCb(X,H)$.

\noindent Following \cite{fuku,AMMP}, given $u\in L^1_\gga(X)$ we say that $u\in BV_\gga(X)$ if
\[\totvar{u}=\sup \left\{ \int_X u \dive_\gga \Phi \, d\gga; \; \Phi \in \FCb(X,H), \; \normH{\Phi}\le 1 \; \forall x\in X\right\}<+\infty.\] 
We will also denote by $|D_\gga u|(X)$ the total variation of $u$. If $u=\chi_E$ is the characteristic function of a set $E$ we will denote $P_\gga(E)$ its total variation and say that $E$ is of finite perimeter if $\Pgamma{E}$ is finite. 
As shown in \cite{AMMP} we have the following properties of $BV_\gga(X)$ functions.

\begin{thm}\label{defBV}
Let $u\in BV_\gga(X)$ then the following properties hold:
\begin{itemize}
\item $D_\gga u$ is a countably additive measure on X with finite total variation and values in $H$ (we will note the space of these measures by $\cM(X,H)$), such that for every $\Phi \in \FCb(X)$ we have:
\[\int_X u \, \partial^*_{h_j} \Phi \; d \gga=- \int_X \Phi d\mu_j \qquad \forall j\in \N\]
where $\mu_j=[h_j,D_\gga u]_H$.
\item $|D_\gga u|(X)=\inf \varliminf \{ \int_X \normH{\nabla_H u_i} d\gga \;: \; u_j \in H^1_\gga(X), \; u_j\rightarrow u \textrm{ in } L^1_\gga(X) \}$.
\end{itemize}
\end{thm}


\begin{prop}
Let $u=v(\Pi_m)$ be a cylindrical function then $u\in BV_\gga(X)$ if and only if $v\in BV_{\gga_m}(\R^m)$. We then have
\[\totvar{u}=\totvarn{v}.\]
\end{prop}

\begin{prop}[Coarea formula \cite{AFP}]\label{procoarea}
If $u\in BV_\gga(X)$ then for every borel set $B\subset X$, 
\begin{equation}\label{coarea}|D_\gga u|(B)=\int_{\R} \Pgamma{\{u>t\},B} \, dt.\end{equation}
\end{prop}

\noindent In Proposition \ref{ehrfunc}, we will need the following extension of Proposition \ref{procoarea}.

\begin{lem}\label{coairg}
For every function $u\in BV_\gga(X)$ and every non-negative Borel function $g$,
\begin{equation}\label{coairg2}
\int_X g(x) \,d|D_\gga u|(x)=\int_{\R}\left(\int_X g(x) \,d|D_\gga \chi_{E_t}|(x)\right)  dt
\end{equation} 
where $E_t:=\{u>t\}$.
\end{lem}

\begin{proof}
The proof of this lemma mimic the standard proof in the Euclidean case \cite[Th.2.2]{CF}.
By \cite[Ch.1,Th.7]{EG} we can write $g$ as
\[
g=\sum_{i=1}^{+\infty} \frac{1}{i}\, \chi^{}_{A_i}
\]
where the $A_i\subset X$ are Borel sets. Using the coarea formula \eqref{coarea}, we then get
\begin{align*}
\int_X g(x) d|D_\gga u|(x)&=\sum_{i=1}^{+\infty} \frac{1}{i} |D_\gga u|(A_i)\\
													&=	\sum_{i=1}^{+\infty} \frac{1}{i}\int_{\R} |D_\gga \chi_{E_t}|(A_i)\, dt\\
													&=\int_{\R} \left(\int_X \sum_{i=1}^{+\infty} \frac{1}{i} \chi_{A_i} d|D_\gga \chi_{E_t}|(x) \right) dt\\
													&=\int_{\R} \int_X g(x)\, d|D_\gga \chi_{E_t}|(x) \, dt.
													\end{align*}
\end{proof}

\noindent In \cite{AMMP} it is also shown that sets with finite Gaussian perimeter can be approximated by smooth cylindrical sets.
\begin{prop}\label{denscyl}
Let $E\subset X$ be a set of finite Gaussian perimeter then there exists smooth sets $E_m \subset \R^m$ such that $\Pi_m^{-1}(E_m)$ converges in $L^1_\gga(X)$ to $E$ and $\Pgamma{ \Pi_m^{-1}(E_m)}=P_{\gga_m}(E_m)$ converges to $\Pgamma{E}$ when $m$ tends to infinity.
\end{prop}

\noindent Note that, for half-spaces, the perimeter can be exactly computed \cite[Cor. 3.11]{AMMP}.
\begin{prop}\label{perhalf}
Let $h=R\hath \in H$ and $c\in \R$ then the half-space $$E=\{x \in X \; :\; \hath(x)\le c\}$$ has perimeter 
\[\Pgamma{E}=\frac{1}{\sqrt{2\pi}} \,e^{-\frac{c^2}{2\normH{h}^2}}.\]
\end{prop}

\section{The Ehrhard symmetrization}\label{rapehr}

The Ehrhard symmetrization has been introduced by Ehrhard in \cite{ehr2} for studying the isoperimetric inequality in a Gaussian setting. 
We recall the definition and the main properties of such symmetrization.

\begin{defin}
We define the functions $\Phi$  and $\alpha$ by 
\[\Phi(x)=\frac{1}{\sqrt{2\pi}}\int_{-\infty}^x e^{-\frac{t^2}{2}} \, dt \qquad \textrm{and} \qquad \alpha(x)=\Phi^{-1}(x).\]

\noindent we then let $\cU(x)=\Phi'\circ\alpha(x)=\frac{1}{\sqrt{2\pi}} e^{-\frac{\alpha^2(x)}{2}}$.
\end{defin}
\noindent Notice that $\Phi(t)$ is the volume of the half-space $\{\hath(x)<t\}$ and that $\cU(x)$ is the perimeter of a half-space of volume $x$.

\begin{lem}\label{lemfin}
Let $\hath_1,\,\hath_2\in\cH$, with $\normH{h_1}=\normH{h_2}=1$, and suppose that there exist $C_1,\,C_2\in\R$ such that 
\[\{\hath_1 <C_1 \}\subset \{\hath_2 < C_2\}.
\]
Then $\hath_1=\hath_2$.
\end{lem}  

\begin{proof}
\noindent Assume by contradiction $\hath_1 \neq \hath_2$, and let $\eta>0$ be such that $|\hath_1-\hath_2|_{\Ldeu} \ge \eta$. 
We shall bound from below by a positive constant the quantity
\[\vgamma{\left\{\hath_1(x)< C_1\right\}\cap\left\{\hath_2(x)\ge C_2\right\}}\]
thus contradicting the inclusion
\[\left\{\hath_1 <C_1 \right\}\subset \left\{\hath_2 < C_2\right\}.\]

\noindent  Letting $h$ be a unitary vector in $H$ orthogonal to $h_1$, we can write
\[h_2=\gl h_1+ \beta h\]
with $\gl^2+\beta^2=1$. 
Up to exchanging $h$ with $-h$,
we can also assume that $\beta\ge0$. We then have $\normH{h_1-h_2}=2(1-\gl)$ and thus $-1\le\gl\le 1-\frac{\eta}{2}$. 

\noindent Let us first suppose that $-1\le \gl\le -\frac 1 2$, then
\[
\left\{\hath_1(x)< \min\left(C_1,-\frac{C_2}{\gl}\right)\right\}\cap\left\{\hath(x)\ge 0\right\}\subset\left\{\hath_1(x)< C_1\right\}\cap\left\{\hath_2(x)\ge C_2\right\}.
\]

\noindent As $\hath_1$ and $\hath$ are orthogonal we have ${\Pi_{\hath_1,\hath}}\sharp \gga=\gga_2$ and thus
\begin{eqnarray*}
\vgamma{\left\{\hath_1(x)< \min(C_1,-\frac{C_2}{\gl})\right\}\cap\left\{\hath(x)\ge 0\right\}} &=& \frac{1}{2} \Phi(\min(C_1,-C_2/\gl))
\\
&\ge& \frac{1}{2}\Phi(\min(C_1,2C_2)).
\end{eqnarray*}
Hence, for $-1\le \gl\le -\frac{1}{2}$,
\[\vgamma{\left\{\hath_1(x)< C_1\right\}\cap\left\{\hath_2(x)\ge C_2\right\}}\ge\frac{1}{2}\,\Phi(\min(C_1,2C_2)).\]

\noindent If now $-\frac{1}{2}\le \gl \le 1-\frac{\eta}{2}$, we can assume that $\eta$ is such that $1-\frac{\eta}{2}\ge\frac{1}{2}$. Let us start by computing the Fourier transform of ${\Pi_{\hath_1,\hath_2}}\sharp \gga$. Denoting by $\tilde{\mu}$ 
the Fourier transform of a measure $\mu$ (see \cite[Sec. 1.2]{boga}) and letting $\Pi:=\Pi_{\hath_1,\hath_2}$, 
for every $(z_1,z_2)\in \R^2$ we have
\begin{align*}
\widetilde{\Pi\sharp \gga}(z_1,z_2)&=\int_{\R^2} e^{i z\cdot x} d \Pi\sharp\gga (x)\\
&=\int_X e^{iz\cdot \Pi(x)} d\gga(x)\\
&=\int_X e^{i [z_1 \hath_1(x)+z_2\hath_2(x)]}d\gga(x)\\
&=\int_X e^{i [(z_1+z_2\gl) \hath_1(x)+z_2\beta \hath(x)]}d\gga(x)\\
&=\int_{\R^2} e^{i [(z_1+z_2\gl)x_1+z_2\beta x_2]}d\gga_2(x_1,x_2)\\
&=\widetilde{\gga_2}(z_1+\gl z_2,\beta z_2)\\ 
&=e^{-\frac{1}{2} [(z_1+\gl z_2)^2+\beta^2 z_2^2]}\\
&=e^{-\frac{1}{2} [z_1^2+z_2^2+2\gl z_1 z_2]}.
\end{align*}
Thus, if we set $K:=\left(\begin{matrix} 1  & \gl\\ \gl &1\end{matrix}\right)$, we have $\widetilde{\Pi\sharp \gga}(z)=e^{-\frac{1}{2} z^t K z}$. It follows that ${\Pi\sharp \gga}$ is a centered Gaussian measure with density $\frac{1}{2\pi\sqrt{\det K}} e^{-\frac{1}{2} z^t K^{-1}z}$ and thus
\[\Pi\sharp\gga (z_1,z_2)=\frac{\sqrt{1-\gl^2}}{2\pi} e^{-\frac{1}{2}[z_1^2+z_2^2-2\gl z_1 z_2]} dz.\]
We now compute 
\begin{align*}\vgamma{\left\{\hath_1(x)< C_1\right\}\cap\left\{\hath_2(x)\ge C_2\right\}}&=\int_X \chi_{\{\hath_1(x)< C_1\}}(x) \chi_{\{\hath_2(x)\ge C_2\}}(x) \,d\gga(x)\\
&=\int_{\R^2} \chi_{\{z_1< C_1\}}(z) \chi_{\{z_2\ge C_2\}} (z) \,d\Pi\sharp \gga(z)\\
&=\int_{-\infty}^{C_1}\int_{C_2}^{+\infty} \frac{\sqrt{1-\gl^2}}{2\pi}\, e^{-\frac{1}{2}[z_1^2+z_2^2-2\gl z_1 z_2]} dz_1 dz_2\\
&\ge \frac{1}{2\pi}\sqrt{\frac{3}{4}}\int_{-\infty}^{C_1}\int_{C_2}^{+\infty}e^{-\frac{1}{2}z_1^2}e^{-\frac{1}{2}z_2^2} 
e^{\gl z_1 z_2} dz_1 dz_2.
\end{align*}
Finally, when $\gl z_1 z_2 \ge 0$, we can bound $e^{\gl z_1 z_2}$ from below by $1$, and when $\gl z_1 z_2 \le 0$ we can bound it form below by 
$e^{-\frac{1}{2} |z_1 z_2|}$ so that we can always bound from below
$$\vgamma{\{\hath_1(x)< C_1\}\cap\{\hath_2(x)\ge C_2\}}$$ by a positive constant. 
\end{proof}

\noindent We now define the Ehrhard symmetrization.
\begin{defin}
Let $E\subset X$ and let $m\in\mathbb N$. 
The Ehrhard symmetral of $E$ along the first $m$ variables is defined as (see Figure \ref{ehrhardsym}):
\[E^*:=\left\{\begin{array}{ll}
\left\{ (x,x_m,x_m^\perp)\in \R^{m-1}\times \R\times X_m^\perp \; : \; x_m<\alpha(\mathbb E_{m-1}\chi_E(x)) \right\}
& {\rm if\ }m>1
\\
\\
\left\{x\in X\; : \;\sprod{x_1^*}{x}< \alpha(\vgamma{E})\right\}
& {\rm if\ }m=1.
\end{array}\right.
\]
\end{defin}

\begin{figure}[ht]
\centering{\input{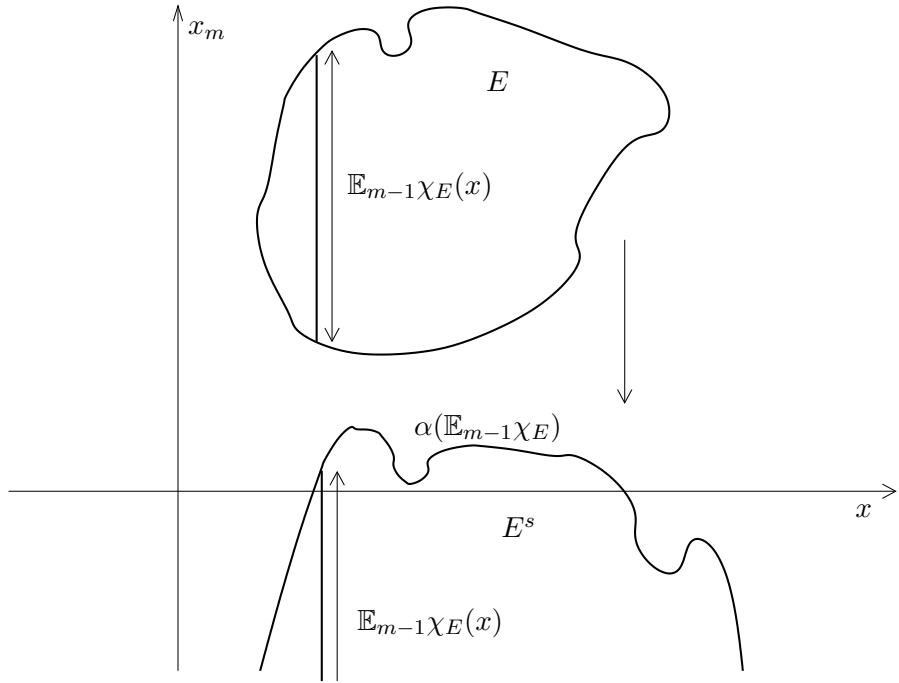}}

     
\caption{The Ehrhard symmetrization.}
\label{ehrhardsym}
\end{figure}

\noindent The interest of this symmetrization is that it decreases the Gaussian perimeter,
while keeping the volume fixed.

\begin{prop}\label{priso}
Let $E$ be a set of finite perimeter and $E^*$ be an Ehrhard symmetral of $E$,
then 
\begin{equation}\label{eqghe}
\gamma(E^*)=\gamma(E),
\end{equation} 
$\mathbb E_{m-1}\chi_{E^*}=\mathbb E_{m-1}\chi_E$ and
\begin{equation}\label{isoper}
\Pgamma{E^*} \le \Pgamma{E} .
\end{equation}
In particular, we have the isoperimetric inequality
\[
\Pgamma{E}\ge \cU(\vgamma{E}),
\]
with equality if and only if $E$ is a half-space. 
\end{prop}
\noindent For the proof we refer to \cite{bobkov,CK},
and to \cite{AMMP} for the extension to infinite dimensions.

\noindent We can also prove a stronger result which is a kind of Bernstein Theorem in this setting.

\begin{prop}\label{probern}
The half-spaces are the only local minimizers of the Gaussian perimeter with volume constraint.
\end{prop}
\begin{proof}
Let $E\subset X $ be a local minimizer of the (Gaussian) perimeter and let $v=\vgamma{E}$. 
This means that, for every $R>0$ and every set $F$ of finite perimeter, with $\vgamma{F}=v$ and $E\gD F \subset B_R$ 
(where $B_R$ denotes the ball of radius $R$ centered at $0$), we have
\[\Pgamma{E}\le \Pgamma{F}.\]

\noindent If $E$ is not an half space then, by Proposition \ref{priso}, there exists $\eta>0$ such that 
\[\Pgamma{E}\ge \cU(v)+\eta.\]
Let $\alpha_R$ be such that
\[\vgamma{E\cap B_R}=\vgamma{\{\sprod{x_1^*}{x}<\alpha_R\}\cap B_R}.\]
We have that $\alpha_R$ tends to $\alpha(v)$ when $R$ goes to infinity and $\Pgamma{\{\sprod{x_1^*}{x}<\alpha_R\}}$ 
tends to $\Pgamma{\{\sprod{x_1^*}{x}<\alpha(v)\}}$. Letting 
\[F_R=\left(\{\sprod{x_1^*}{x}<\alpha_R\}\cap B_R\right)\cup \left(E\cap B_R^c\right)
\]
we get 
 \begin{align*}
 \cU(v)+\eta\le\Pgamma{E}\le\Pgamma{F_R}&\le \Pgamma{\{\sprod{x_1^*}{x}<\alpha_R\}\cap B_R}+\Pgamma{E\cap B_R^c}\\
 													&\le \Pgamma{\{\sprod{x_1^*}{x}<\alpha_R\}}+\Pgamma{B_R}+\Pgamma{E\cap B_R^c}\\
													&\le \Pgamma{\{\sprod{x_1^*}{x}<\alpha(v)\}}+ \eps(R)\\
													&=\cU(v)+\eps(R),
													\end{align*}
where we used various time the inequality (see \cite{giusti}) 
$$\Pgamma{E\cup F}+\Pgamma{E\cap F} \le \Pgamma{E}+\Pgamma{F}$$ 
and where $\eps(R)$ is a function which goes to zero when $R$ goes to infinity. We thus found a contradiction.													
\end{proof}

\begin{remarque} \rm
In the Euclidean setting, half-spaces are the only local minimizers of the perimeter only in dimension lower than 8 (see \cite{giusti}). Notice also that if we drop the volume constraint, half spaces are no longer local minimizers for the Gaussian perimeter,
since there are no nonempty local minimizers.
\end{remarque}

\noindent In the sequel we will also need another transformation which from a finite dimensional function gives an 
Ehrhard symmetric set whose sections have volume prescribed by the original function. More precisely:
\begin{defin}
Given a measurable function $v:\R^m\to [0,1]$, we define its Ehrhard set $ES_m(v)\subset X$ by
\[
ES_m(v):=\left\{(x,x_{m+1},x_{m+1}^\perp)\in \R^m\times \R \times X_{m+1}^\perp \; : \; x_{m+1}<\alpha(v(x))\right\}.
\]
Given a measurable cylindrical function $u:X\to [0,1]$
depending only on the first $m$ variables, that is, $u=v\circ\Pi_m$ for some $v:\R^m\to [0,1]$, we set
\[
ES_m(u):=ES_m(v).
\] 
\end{defin}

\noindent The link between Ehrhard sets and Ehrhard symmetrization is the following:

\begin{prop}
Let $E$ be a set of finite perimeter and $E^*$ be its Ehrhard symmetrization with respect to the first $(m+1)$ variables, then
\[E^*=ES_m(\bEm(\chi_E)).\]
\end{prop}

\noindent In the next proposition we compute the perimeter of Ehrhard sets. It slightly extends a result in \cite{CFMP}.
\begin{prop}\label{egalite}
Let $u\in BV_{\gamma_m}(\R^m)$ with $0\le u\le 1$, then 
\[
P_{\gamma}(ES_m(u))=\int_{\R^m} \sqrt{\cU(u)^2+|D_{\gamma_m} u|^2} \,d\gga_m
\]
where $$\int_{\R^m} \sqrt{\cU(u)^2+|D_{\gamma_m} u|^2} d\gga_m=\int_{\R^m} \sqrt{\cU(u)^2+|\nabla u|^2} \,d\gga_m+|D^s_\gga u|(X)$$ 
and $D_\gga u= \nabla u\, \gga + D^s_\gga u$
is the Radon-Nikodym decomposition of $D_\gga u$.
\end{prop}

\begin{proof}
By \cite[Th. 4.3]{CFMP} the result holds for $u\in H^1_{\gga_m}(\R^m)$. 
We will show by approximation that the same holds for $u\in BV_{\gamma_m}(\R^m)$.

\noindent Let $E=ES_m(u)$, then we can find sets $E_n$ such that 
$\gamma(E_n\Delta E)\to 0$ and 
$P_{\gga}(E_n)\to P_{\gga}(E)$ as $n\to +\infty$,
and all the $E_n$ have smooth boundary and are Ehrhard symmetric.  
Thus, for every $n\in\mathbb N$, there exists a smooth function $u_n$ such that $0\le u_n\le 1$, $ E_n=ES_m (u_n)$, 
$u_n\to u$ in $L^1_{\gamma_{m}}(\R^{m})$, and  
\[
P_{\gga}(E_n)=\int_{\R^m} \sqrt{\cU(u_n)^2+|D_{\gamma_m} u_n|^2}\, d\gga_m.
\]

\noindent Since, by Proposition \ref{dual}, the functional $\int_{\R^m} \sqrt{\cU(u)^2+|D_{\gamma_m} u|^2} d\gga_m$ is lower semicontinuous 
in $L^1_{\gamma_{m}}(\R^{m})$, we get
\begin{eqnarray*}
P_{\gga}(E) &=& \lim_{n\to\infty}P_{\gga}(E_n)
\\
&=& \lim_{n\to\infty} \int_{\R^m} \sqrt{\cU(u_n)^2+|D_{\gamma_m} u_n|^2}\, d\gga_m 
\\
&\ge& \int_{\R^m} \sqrt{\cU(u)^2+|D_{\gamma_m} u|^2}\, d\gga_m.
\end{eqnarray*}

\noindent The other inequality follows as in \cite{CFMP}. Let $\widetilde E=\Pi_{m+1}(E)\subset\R^{m+1}$ 
and observe that $\gamma_{m+1}(\widetilde E)=\gamma(E)$ and $P_{\gga_{m+1}}(\widetilde E)=P_\gamma(E)$.
 By Vol'pert Theorem 
\cite[Th. 3.108]{AFP} there exists a set $B\subset \R^m$ 
such that for every $x\in B$, $\nu^{\widetilde E}_{m+1}(x,\alpha(u_E(x)))$ 
exists and is not equal to zero, where $\nu^{\widetilde E}_{m+1}$ denotes the 
last coordinate of the unit external normal to $\partial^* \widetilde E$. 
By \cite[Lemma 4.4]{CFMP}, $\gga_m$-almost every  $x\in B$ 
is a point of approximate differentiability for $u$.  By Lemma 4.5 and 4.6 of \cite{CFMP} we then have
\begin{align*}
P_{\gga_{m+1}}(\widetilde E)&=P_{\gga_{m+1}}(\widetilde E,B\times \R)+P_{\gga_{m+1}}(\widetilde E,B^c\times \R)\\
								&\le \int_{B} \sqrt{\cU(u)^2+|\nabla u|^2}d\gga_m + \int_{B^c} |D_{\gga_m} u| +\int_{B^c} \cU(u)\, d\gga_m\,.
								\end{align*}
As $\gga_m(B^c)=0$, we find that
\[\int_{B} \sqrt{\cU(u)^2+|\nabla u|^2}d\gga_m + \int_{B^c} |D_{\gga_m} u|=
\int_{\R^m} \sqrt{\cU^2(u)+|\nabla u|^2} d\gga_m + |D^s_{\gga_m} u|(\R^m)\]
and thus $P_\gamma(E)=P_{\gga_{m+1}}(\widetilde E) \le \int_{\R^m} \sqrt{\cU(u)^2+|D_{\gamma_m} u|^2} d\gga_m$.
\end{proof}

\noindent The last transformation that we consider is the analog of the Schwarz symmetrization in the Gaussian setting,
and was first introduced by Ehrhard in \cite{ehr}.

\begin{defin}
Let $u\in X\to\R$ be a measurable function and let $m \in \N$ be fixed. We define the $m$-dimensional Ehrhard symmetrization $u^*$ of $u$ 
as follows:
\begin{itemize}
\item for all $t\in\R$ we let $E_t^*$ be the Ehrhard symmetrization of $E_t:=\{u>t\}$ with respect to the first $m$ variables;
\item we let $u^*(x):=\inf \{ t \; :\; x\in E_t^*\}$.
\end{itemize}
\end{defin}

\noindent As \eqref{eqghe} implies $\gamma(\{u^*>t\})=\gamma(\{u>t\})$ for all $t\in\R$, 
from the Layer Cake formula it follows that, if $u\in L^2_\gamma(X)$, 
then $u^*\in L^2_\gamma(X)$ and
\begin{equation}\label{layer}
\int_X |u^*|^2 d\gga=\int_X |u|^2\,d\gga\,.
\end{equation}
Indeed, we have
\begin{eqnarray*}
\int_X |u|^2 d\gga &=& 2\int_0^{+\infty} t \, \gga ( \{u >t\})\, dt 
- 2 \int^0_{-\infty} t \, \gga ( \{u <t\})\, dt 
\\
&=& 
2\int_0^{+\infty} t \, \gga ( \{u^* >t\})\, dt
- 2 \int^0_{-\infty} t \, \gga ( \{u^* <t\})\, dt
\\
&=&\int_X |u^*|^2 d\gga.
\end{eqnarray*}

\begin{lem}\label{lemlam}
Let $u,v:X\to [0,+\infty)$ belonging to $L^2_\gamma(X)$, 
then
\begin{equation}\label{diffstar}
\|u^*-v^*\|_{L^2_\gamma(X)} \le \|u-v\|_{L^2_\gamma(X)}.
\end{equation}
\end{lem}

\begin{proof}
\noindent The proof is a straightforward adaptation of the 
analogous proof for the Schwarz symmetrization \cite[Th. 3.4]{liebloss}. 

\noindent Recalling \eqref{layer} with $p=2$, we have only to show that 
\begin{equation}\label{uv}\int_X u v d \gga \le \int_X u^* v^* d \gga.\end{equation}

\noindent Again by the Layer Cake formula we have
\[\int_X u v d \gga=\int_0^{+\infty} \int_0^{+\infty} \int_X \chi_{\{u>t\}}(x) \chi_{\{v>s\}}(x) d\gga(x) \, dt \, ds.\]

\noindent Thus \eqref{uv} would follow from the same inequality for sets, that is, 
\begin{equation}\label{diffset}
\vgamma{A\cap B} \le \vgamma{A^* \cap B^*}. 
\end{equation}
Let $x_m \in \R^m$ and assume that 
\[\int_{X_m^\perp} \chi_A (x_m+y) d\gga_m^\perp(y) \ge \int_{X_m^\perp} \chi_B (x_m+y) d\gga_m^\perp(y)\]
then by definition of the Ehrhard symmetrization we have $$B^* \cap (x_m +X_m^\perp) \subset A^* \cap (x_m +X_m^\perp)$$ and therefore
\begin{align*}\int_{X_m^\perp} \chi_{A^*} (x_m+y)\chi_{B^*}(x_m+y) d\gga_m^\perp(y)&=\int_{X_m^\perp} \chi_{A^*} (x_m+y) d\gga_m^\perp(y)\\
&=\int_{X_m^\perp} \chi_{A} (x_m+y) d\gga_m^\perp(y)\\
&\ge \int_{X_m^\perp} \chi_{A} (x_m+y)\chi_{B}(x_m+y) d\gga_m^\perp(y)
\end{align*}
This inequality also holds if $\int_{X_m^\perp} \chi_B (x_m+y) d\gga_m^\perp(y) \ge \int_{X_m^\perp} \chi_A (x_m+y) d\gga_m^\perp(y)$ so that finally
\begin{align*}
\vgamma{A^* \cap B^*}&=\int_{X_m} \int_{X_m^\perp}  \chi_{A^*} (x+y)\chi_{B^*}(x+y) d\gga_m^\perp(y) d\gga_m(x)\\
& \ge \int_{X_m}  \int_{X_m^\perp} \chi_{A} (x+y)\chi_{B}(x+y) d\gga_m^\perp(y) d\gga_m(x)\\
&=\vgamma{A\cap B}\end{align*}
which gives \eqref{diffset}.
\end{proof}

\noindent As for the Schwarz symmetrization, a P\'olya-Szeg\"o principle holds for the Ehrhard symmetrization.

\begin{prop}\label{ehrfunc}
Let $u\in H^{1}_\gga(X)$, let $m \in \N$ and let $u^*$ be the 
$m$-dimensional Ehrhard symmetrization of $u$.
Then $u^*\in H^{1}_{\gga}$ and 
\begin{equation}\label{polya} 
\int_{X} \normH{\nabla_H u^*}^2 \,d\gga \le \int_X \normH{\nabla_H u}^2 \,d\gga.
\end{equation}
Moreover, if $m=1$ and equality holds in \eqref{polya}, then 
$$
u=\tilde{u}\left(\hath(x)\right) \qquad {\rm for\ some\ } \hath \in \cH\,,
$$ 
and $\hath$ can be chosen to be a unitary vector.
\end{prop}

\begin{proof}

\noindent 
In \cite[Th. 3.1]{ehr}, inequality \eqref{polya} is proven for Lipschitz functions, in finite dimensions. 
We extend it by approximation to Sobolev functions. 

\noindent We can assume $u\ge 0$, since we have 
$(u^\pm)^*=(u^*)^\pm$,
where $u^\pm, (u^*)^\pm$ denote the positive and negative part of $u$ and $u^*$, respectively.

\noindent Let $u_n \in \FCc(X)$ be positive functions converging to $u$ in $H^1_\gga(X)$, 
then by \eqref{diffstar}, $u_n^*$ converges to $u^*$ in $\Ldeu$ and thus by the lower semicontinuity of the $H^1_\gga(X)$ norm we have
\[\int_X \normH{\nabla_H u^*}^2\le \varliminf_{n\to \infty} \int_X \normH{\nabla_H u_n^*}^2\le\varliminf_{n\to \infty} \int_X \normH{\nabla_H u_n}^2= \int_X \normH{\nabla_H u}^2.\]

\smallskip

\noindent We now turn to the equality case for one-dimensional symmetrizations. For this we closely follow \cite{CK} and
 give an alternative proof of \eqref{polya}, based on ideas of Brothers and Ziemer \cite{BZ} for the Schwarz symmetrization.

\noindent Let $u\in H^1_{\gga}(X)$ and $\mu(t)=\gga( \{u>t\})=\gga( \{u^*>t\})$. 
By the coarea  formula \eqref{coairg2}, for all $t\in\R$ we have
\[\mu(t)= \gga(\{u>t\} \cap \{\nabla_H u=0\})+\int_t^{+\infty} \left(\int_{\{\nabla_H u\ne 0\}} \frac{1}{\normH{\nabla_H u}}\;  d|D_\gga \chi_{E_\tau}| \right) d\tau.\] 
Hence
\begin{equation}\label{derivmu}
-\mu'(t)\ge \int_{\{\nabla_H u\ne 0\}} \frac{1}{\normH{\nabla_H u}} \,d|D_\gga \chi_{E_t}|  \qquad  {\rm for\ a.e.\ } t\in\R.
\end{equation}

\noindent Since $u^*$  is a function depending only on one variable, arguing as in \cite{CF}  we get 
$$\frac{d}{dt}\,  \gga(\{u^*>t\} \cap \{\nabla_H u^*=0\})=0 \qquad   {\rm for\ a.e.\ }  t\in\R.
$$ 
As $u^*$ is monotone we have that $\normH{\nabla_H u^*}$ is constant on $\{u^*=t\} \cap \{\nabla_H u^* \neq 0\}$. 
Observe also that, being $u^*$ one-dimensional, $\{u^*=t\}$ has a well defined meaning. We thus find:
\begin{equation*}
-\mu'(t)=  \frac{P_{\gga}(\{u^*>t\})}{|\nabla_H u^*|_{\{u^*=t\}}}  \qquad   {\rm for\ a.e.\ }  t\in\R,
\end{equation*}
which implies, recalling \eqref{derivmu},
\begin{equation}\label{egalmustar}
\frac{P_{\gga}(\{u^*>t\})}{|\nabla_H u^*|_{\{u^*=t\}}}  \ge \int_{\{\nabla_H u\ne 0\}} \frac{1}{\normH{\nabla_H u}} \,d|D_\gga \chi_{E_t}| 
\qquad   {\rm for\ a.e.\ }  t\in\R.
\end{equation}

\noindent Let us note that as in \cite[Lem. 4.2]{CK}, using \eqref{coairg2} with $g=\chi_{\{\nabla_H u =0\}}$ we find
\[\int_X \chi_{\{\nabla_H u =0\}} \normH{\nabla_H u} d\gga =0=\int_{\R} \int_X \chi_{\{\nabla_H u =0\}} d|D_\gga \chi_{E_t}|(x)\; dt\]
and thus for almost every $t \in \R$, 
\[\int_X \chi_{\{\nabla_H u =0\}} d|D_\gga \chi_{E_t}|(x)=0.\]
This shows that for almost every $t\in \R$, $\nabla_H u(x) \neq 0$ for $|D_\gga \chi_{E_t}|$-almost every $x\in X$ and thus 
\begin{equation}\label{egalmustice}
\int_{\{\nabla_H u\ne 0\}} \frac{1}{\normH{\nabla_H u}} d|D_\gga \chi_{E_t}|(x)=\int_{X} \frac{1}{\normH{\nabla_H u}} d|D_\gga \chi_{E_t}|(x) \qquad {\rm for\ a.e.\ } t\in \R.
\end{equation}

\noindent By \eqref{coairg2}, \eqref{isoper},\eqref{egalmustar} and \eqref{egalmustice}, we eventually get
\[\begin{array}{clc}
\displaystyle \int_{X} |\nabla_H u^*|^2 d\gga& \displaystyle =\int_{\R} |\nabla_H u^*|_{\{u^*=t\}} P_\gga(\{u^*>t\}) dt& \\[14pt]
																 \\
&\displaystyle =\int_{\R} \frac{P_{\gga}(\{u^*>t\})^2}{ \left(\frac{P_\gga(\{u^*>t\})}{|\nabla_H u^*|_{\{u^*=t\}}}\right)}dt&\\[18pt]
																 \\
&\displaystyle \le \int_{\R} \frac{P_{\gga}(\{u>t\})^2}{\int_{X} \frac{1}{\normH{\nabla_H u}} d|D_\gga \chi_{E_t}|(x)}\; dt 
																\\[14pt] 
																 \\
&\displaystyle \le \int_{\R} \int_{X} \normH{\nabla_H u} \; d|D_\gga \chi_{E_t}|(x)\; dt
																\\[14pt]
																&=\displaystyle \int_{X} \normH{\nabla_H u}^2 d\gga\,. 
																\end{array}\]
As a consequence, if equality holds in \eqref{polya}, then 
equality holds in the Gaussian isoperimetric inequality, that is,
\[P_{\gga}(u>t)=P_{\gga}(u^*>t)\qquad   {\rm for\ a.e.\ }  t\in\R.\]
This implies that almost every level-set of $u$ is a half-space, i.e. for almost every $t\in\R$ there exists $\hath_t \in \cH$ such that $\{u>t\}=\{\hath_t < \alpha(\mu(t))\}$, and without loss of generality we can assume that $\normH{h_t}=1$. Such half-spaces being nested, 
by Lemma \ref{lemfin} we have that $\hath_t$ does not depend on $t$ and thus $u(x)=v(\hath (x))$. 
\end{proof}

\begin{remarque} \rm

We notice that the fact that equality in \eqref{polya} implies that $u$ is one-dimensional is a specific feature of the
Gaussian setting, and the analogous statement does not hold 
for the Schwarz symmetrization in the Euclidean case \cite{BZ}.
Indeed, this property is a consequence of the fact that Gaussian measures, differently from the Lebesgue measure,
are not invariant under translations.
\end{remarque}

\section{Relaxation of perimeter}\label{secrelax}

In this section we compute the relaxation of the perimeter functional
 \[F(u):=\left\{\begin{array}{ll}
\Pgamma{E} \qquad& \textrm{if } u=\chi_E\\[8pt]
+\infty \qquad & \textrm{otherwise}
\end{array}\right.\]
with respect to the weak $\Ldeu$-topology. The fact that $F$ is not lower semicontinuous can be easily checked by taking the sequence ${E_n}=\{ \sprod{x_n^*}{x}<0\}$. Indeed,
the characteristic functions of these sets weakly converge to the constant function $1/2$, which is not a characteristic function,
while the perimeter of $E_n$ is constantly equal to $1/\sqrt{2\pi}$.

We will show that the relaxation of $F$ is equal to
\[
\barf(u):=\left\{\begin{array}{ll}
\displaystyle \int_X \sqrt{\cU^2(u)+|D_\gga u|^2}d \gga \qquad& \textrm{if } 0\le u\le 1 \quad \gga - a. e.\\[8pt]
+\infty \qquad & \textrm{otherwise}
\end{array}\right.
\]
where 
\[ 
\int_X \sqrt{\cU^2(u)+|D_\gga u|^2}d \gga=\int_X \sqrt{\cU^2(u)+\normH{\nabla_H u}^2}d \gga + |D^s_\gga u|(X)
\] 
with $D_\gga u= \nabla_H u d \gga + D^s_\gga u$. 
Observe that the functional $\barf$ already appears in the seminal work of Bakry and Ledoux \cite{bakryledoux} 
and in the earlier work of Bobkov \cite{bobkov} in the context of log-Sobolev inequalities. This functional has been also studied in \cite{CK}. See also \cite[Remark 4.3]{AMMP} where it appears in a setting closer to ours.

Let us first recall the definition of the lower semicontinous envelope of a function (see \cite{DM} for more details).
\begin{defin}
Let $X$ be a topological vector space. For every function $F: X\rightarrow \overline{\R}$, its lower semicontinuous envelope (or relaxed function) is the greatest lower semicontinuous function that lies below $F$.
\end{defin}
\noindent When $X$ is a metric space, the following caracterization holds.
\begin{prop}
Let $X$ be a metric space. For every function $F:X\rightarrow \overline{\R}$, and every $x\in X$, the relaxed function $\barf$ is given by
\[\barf(x)=\inf \left\{\varliminf_{n\to \infty} F(x_n) \; :\; x_n\to x \right\}.\] 
\end{prop}

We now show a representation formula for $\barf$ which is reminiscent of the definition of 
the total variation and of the nonparametric area functional (see \cite{giusti}).
We start with a preliminary result.

\begin{lem}\label{lemdual}
Let $g \in L^\infty (X)$ with $g \ge 0$, let $\mu \in \cM(X,H)$, and define
\[
\tilde f(g,\mu):=\sqrt{g^2+\normH{h}^2}\, d\gga + |\mu^s|\,,
\]
where $\mu=h\,\gamma+\mu^s$.
There holds
\begin{equation}\label{eqtilde}
\tilde f(g,\mu)(X)=\sup_{ \stackrel{\Phi\in L^1_\mu(X,H)}{ \xi \in L^1_\mu(X)}} 
\left\{ \int_X [\Phi, d\mu]_H +\int_X g\,\xi \, d\gga : \; 
\normH{\Phi}^2+|\xi|^2\le 1\ {\rm a.e.\ in\ } X\right\}.
\end{equation}
\end{lem}

\begin{proof}
The proof is adapted from \cite{DemTem}.

Notice first that, for $(\gl,p)\in \R \times H$, the function 
$f(\gl,p):=\sqrt{\gl^2+\normH{p}^2}$ defines a norm on the product space $\R\times H$. 
Moreover,
if we let ${f}_\gl(p):=\sqrt{\gl^2+\normH{p}^2}$, then the convex conjugate of 
${f}_\gl$ is ${f}^*_\gl(\Phi)=-\gl\sqrt{1-\normH{\Phi}^2}$. We divide the proof into three steps.

\smallskip

\noindent{\it Step 1.}
Let 
\[
M(g,\mu)=\sup_{\Phi\in L^1_\mu(X,H)} \left\{ \int_X [\Phi, d\mu]_H +\int_X g\sqrt{1-\normH{\Phi}^2} \, d\gga 
:\;\normH{\Phi}\le 1 \ {\rm a.e.\ in\ } X\right\}.
\]
We will show that 
\begin{equation}\label{eqM}
M(g,h\gga)=\int_X f(g,h)d\gga.
\end{equation}
By definition of convex conjugate, it is readily checked that $M(g,h\gga)\le \int_X f(g,h)d\gga$.
We thus turn to the other inequality. By definition of the Bochner integral, 
for every $\gd>0$, there exists $h_i \in H$ and $A_i \subset X$ with $A_i$ disjoints Borel sets and $i\in[1,m]$ such that if we set 
\[\theta= \sum_{i=1}^{m} \chi_{A_i} h_i\]
then $|\theta-h|_{L^1_\gga} \le \gd$. Analogously there exists $\eta_i\in X$ such that setting
\[\tilde{g}=\sum_{i=1}^{m} \chi_{A_i} \eta_i\]
we have $|\tilde{g}-g|_{L^1_\gga} \le \gd$. 
By the observation at the beginning of the proof and the triangle inequality we get
\[|f(\tilde{g},\theta)-f(g,h)|\le f(\tilde{g}-g,\theta-h)|\le |\tilde{g}-g|+\normH{\theta-h}.\]

\noindent For every $i$, by definition of convex conjugate, there exists $\xi_i \in H$ with $\normH{\xi_i}\le1$ such that
\[
f(\eta_i,h_i)\le [\xi_i, h_i]_H + \eta_i \sqrt{1-\normH{\xi_i}^2} +\delta.\]

\noindent From this, setting $\Phi=\sum_{i=1}^m \chi_{A_i} \xi_i$ we have
\begin{align*}
\int_X f(g,h) d\gga &\le \int_X f(\tilde{g},\theta) d\gga + 2\delta \\
										&= \sum_{i=1}^{m} \int_{A_i} f(\eta_i,h_i) d\gga + 2\delta\\
										&\le \sum_{i=1}^{m} \int_{A_i} [\xi_i, h_i]_H + \eta_i \sqrt{1-\normH{\xi_i}^2} d\gga +3\delta\\
					          &=\int_X [\Phi, h]_H + \tilde{g}\sqrt{1-\normH{\Phi}^2} d\gga + 3 \delta. 
										\end{align*}

\noindent  Since $\bbar\, \tilde{g}\sqrt{1-\normH{\Phi}^2}-g\sqrt{1-\normH{\Phi}^2}\, \bbar\le |\tilde{g}-g|$ we find
\begin{align*}
\int_X f(g,h) d\gga &\le \int_X \Phi\cdot h -g\sqrt{1-\normH{\Phi}^2} d\gga + 4\delta\\
											&\le M(g, h \gga) + 4\delta.				
\end{align*}
Since $\delta$ is arbitrary we have $M(g,h\gga)=\int_X f(g,h)d\gga$.

\smallskip

\noindent{\it Step 2.} The proof proceeds exactly as in \cite{DemTem} and we only sketch it. 
Recalling \eqref{eqM}, it remains to show that
\[
M(g,h \gga+\mu^s)=M(g,h \gga)+ |\mu^s|(X).
\] 
One inequality is easily obtained, since
\begin{align*}
M(g,h \gga+\mu^s)&=\sup_\Phi \int_X [\Phi, h]_H d\gga +\int_X \Phi\cdot d\mu^s+\int_X g(x)\sqrt{1-\normH{\Phi}^2} d\gga\\
												&\le \left(\sup_\Phi \int_X [\Phi, h]_H d\gga +\int_X g(x)\sqrt{1-\normH{\Phi}^2} d\gga\right) +\int_X |d\mu^s|\\
												&=M(g,h \gga)+ |\mu^s|(X).
												\end{align*}
For the opposite inequality, let $\delta>0$ be fixed then there exists $\Phi_1$ and $\Phi_2$ such that
\begin{align*}
M(g,h\gga) &\le \int_X [\Phi_1, h]_H d\gga + \int_X g(x)\sqrt{1-\normH{\Phi_1}^2} d\gga + \delta\\
|\mu^s|(X)&\le \int_X [\Phi_2, d\mu^s]_H + \delta.
\end{align*}
Taking $\Phi$ equal to $\Phi_2$ on a sufficiently small neighborhood of the support of $\mu^s$ and equal to $\Phi_1$ outside this neighborhood, we get
\begin{align*}
M(g,h\gga)+|\mu^s|(X)&\le\int_X [\Phi, h]_H d\gga + \int_X g(x)\sqrt{1-\normH{\Phi}^2} d\gga + \int_X [\Phi, d\mu^s]_H + C\delta\\
&\le M(g, h \gga+ \mu^s)+C\delta
\end{align*}
which gives the opposite inequality.

\smallskip

\noindent{\it Step 3.} In order to conclude the proof, it is enough to notice that 
for every $\Phi \in L^1_\mu(X,H)$, with $\normH{\Phi}\le 1$, we have 
\begin{eqnarray*}
&&\sup_{\xi \in L^1_\mu(X)} 
\left\{ \int_X [\Phi, d\mu]_H +\int_X g\,\xi \, d\gga : \; 
\normH{\Phi}^2+|\xi|^2\le 1\ {\rm a.e.\ in\ } X\right\}
\\
&&\quad =
\int_X [\Phi, d\mu]_H +\int_X g\sqrt{1-\normH{\Phi}^2} \, d\gga.
\end{eqnarray*}
\end{proof}

\begin{prop}\label{dual}
Let $u \in BV_\gga(X)$ then
\begin{equation}\label{eqqa}
\barf(u)=\sup_{ \stackrel{\Phi\in \FCb(X,H)}{ \xi \in \FCb(X)}} \left\{ \int_X \left(u \dive_\gga \Phi + \cU(u) \xi \right) d\gga \,: \quad \normH{\Phi(x)}^2+|\xi(x)|^2\le 1 \; \forall x \in X \right\}.
\end{equation}
\end{prop}

\begin{proof}
We apply Lemma \ref{lemdual} with $\mu=Du$ and $g=\cU(u)$.
Since $\mu$ is tight \cite{AMMP},
the space $\FCb(X,H)$ is dense in $L^1_\mu(X,H)$ so that we can restrict the supremum in \eqref{eqqa}
to smooth cylindrical functions $\Phi,\,\xi$.
\end{proof}

\begin{remarque}\rm
Since $\cU$ is concave, the duality formula \eqref{eqqa} is not sufficient to prove that $\barf$ is lower semicontinuous for the weak $\Ldeu$-topology. 
It shows however the lower-semicontinuity of $\barf$ in the strong $\Ldeu$-topology.
\end{remarque}

\noindent We now prove that $\barf$ is the lower semicontinuous envelope of $F$.

\begin{thm}\label{relax}
$\barf$ is the relaxation of $F$ in the weak  $\Ldeu$-topology. 
\end{thm}

\begin{proof}
Let us first notice that $F$ takes finite values only on functions of the closed unit ball of $\Ldeu$ which is metrizable for the weak convergence. 
Therefore the relaxation and the sequential relaxation in the weak topology of $\Ldeu$ coincide. 

Let $\chi_{E_n}$ be a sequence of sets weakly converging in $\Ldeu$ to $u \in BV_\gga(X)$, with uniformly bounded perimeter. We shall show that 
\[
\varliminf_{n\to\infty}  P_\gga({E_n})\ge \barf(u).
\]
Notice that, by weak convergence, we necessarily have $0\le u\le 1$ a.e. on $X$. 

\noindent For all $n\ge 1$ and $k\ge2$, we let $E_n^k$ be the Ehrhard symmetral of $E_n$ with respect to the first $k$ variables. 
Recalling the notation of Section \ref{rapehr}, we have 
\[
\Pgamma{E_n^{k+1}}\le \Pgamma{E_n} \qquad \textrm{and} \qquad {E_n^{k+1}}=ES_k\left(\E_k\chi_{E_n}\right).
\]
As $\displaystyle \totvar{\E_k(\chi_{E_n})}\le \Pgamma{{E_n}}$ and $\E_k(\chi_{E_n})$ depends only on the first $k$ variables,
by the compact embedding of 
$BV_{\gga_k}(\R^k)$ into $L^1_{\gamma_k}(\R^k)$ we can extract a subsequence from 
$\E_k(\chi_{E_n})$ which converges strongly to $u^k:=\E_k(u)$. 
{}From this we get that $E^{k+1}_n=ES_k(\E_k\chi_{E_n})$ tends strongly to $E^{k+1}:=ES_k(u^k)$. 
By the lower semicontinuity of the perimeter we then have
\[
\varliminf_{n \to \infty} \Pgamma{E_n} \ge \varliminf_{n\to \infty} \Pgamma{E^{k+1}_n}\ge \Pgamma{E^{k+1}}.
\] 
For every $\gphi \in \FCb (X)$, with $\gphi$ depending only of the $j\le k$  
first variables, there holds 
\[
\int_X \chi_{E_{k+1}} (x) \gphi(x) d\gga (x)=\int_X u_k(x) \gphi(x) d\gga (x)=\int_X u (x) \gphi(x) d\gga (x),
\]
which implies that the sequence $\chi_{E_{k+1}}$ tends weakly to $u$. 
In order to conclude the proof
it remains to show that
\[
\lim_{k\to \infty} \Pgamma{E^{k+1}}=\barf(u).
\]
Notice that, by Proposition \ref{egalite}, there holds
\[
\Pgamma{E^{k+1}}=\barf(u^k).
\] 
For every $\Phi \in \FCb(X,H)$ and $\xi \in \FCb(X)$, depending on the first $k$ variables and such that the range of $\Phi$ is included in $H_k$, 
by Proposition \ref{dual}, we have
\[
\int_{X} \left(u^k \dive_\gga \Phi + \cU(u^k) \xi\right) d\gga=\int_{X} \left(u \dive_\gga \Phi + \cU(u) \xi \right)d\gga\le \barf(u).
\]
Taking the supremum in $\Phi,\,\xi$ and recalling \eqref{eqqa},
we then get 
$$
\barf(u^k)\le \barf(u) \qquad {\rm for\ all\ }k.
$$ 
Repeating the same argument with $u^{k+1}$ instead of $u$, we obtain that $\barf(u^k)$ is nondecreasing in $k$. 
Therefore there exists $\ell \ge 0$ such that
\[
\lim_{k\to \infty} \barf(u^k)=
\lim_{k\to \infty} \Pgamma{E^{k+1}}=\ell \le \barf(u).
\]
Assume by contradiction that $\ell<\barf(u)$. Then there exists $\delta>0$ such that $\barf(u^k)\le \barf(u)-\delta$ for all $k$, hence
there exist $N\in\mathbb N$, $\Phi \in \FCb(X,H)$ and $\xi \in \FCb(X)$, depending only on the first $N$ variables, such that
\[
\int_X \left(u^k \dive_\gga \Phi + \cU(u^k) \xi\right) d\gga 
\le \barf(u^k)\le\barf(u)-\delta\le \int_X \left(u \dive_\gga \Phi + \cU(u)\xi\right) d \gga -\frac{\delta}{2},
\]
but for $k>N$ we have
$$ 
\int_X \left(u^k \dive_\gga \Phi + \cU(u^k) \xi\right) d\gga
=\int_X \left(u \dive_\gga \Phi + \cU(u)\xi\right) d \gga$$ 
which leads to a contradiction.
\end{proof}

\begin{remarque}\rm
Theorem \ref{relax}  
provides an example of a nonconvex functional, namely $\barf$, 
which is lower semicontinuous for the weak $\Ldeu$-topology.
We also know that semicontinuity does not holds for general functional of the form
\[J(u)=\int_X f(u,D_\gga u) d\gga\]
since if we take for instance $f(u,p):= \sqrt{g^2(u)+|p|^2}$ with $g$ such that $g(1/2)>\cU(1/2)$ and $g(0)=g(1)=0$, then,  
letting $u_n:=\{\sprod{x_n^*}{x}<0\}$, we have $u_n\rightharpoonup u= 1/2$ weakly in $\Ldeu$, so that 
$$
J(u)=g\left(\frac{1}{2}\right)>
\cU\left(\frac{1}{2}\right)=\frac{1}{\sqrt{2\pi}}=\varliminf_{n\to\infty} J(u_n).
$$ 
One could wonder what are the right hypotheses for a functional of this form to be lower semicontinuous
with respect to the weak topology.
\end{remarque}

\section{$\gG$-limit for the Modica-Mortola functional}\label{MMM}
Let us briefly recall the definition of $\gG$-convergence.
We refer to \cite{DM} for a comprehensive treatment of the subject.

\begin{defin}
Let $X$ be a topological space, and 
let $F_n:X\to \overline{\R}$ be a sequence of functions. The $\gG$-lower limit and the $\gG$-upper limit of the sequence $F_n$ is defined as
\begin{align*}
(\gG-\varliminf_{n\to \infty} F_n)(x)&=\sup_{U\in \mathcal{N}(x)} \, \varliminf_{n\to \infty} \, \inf_{y\in U} F_n(y)\\
(\gG-\varlimsup_{n\to \infty} F_n)(x)&=\sup_{U\in \mathcal{N}(x)} \, \varlimsup_{n\to \infty} \, \inf_{y\in U} F_n(y)
\end{align*}
where $\mathcal{N}(x)$ denotes the set of all open neighbourhoods of $x$ in $X$. When the $\gG$-lower limit and the $\gG$-upper limit coincide, 
we say that the sequence $F_n$ $\gG$-converges.
\end{defin}

\noindent As for the relaxation, if $X$ is a metric space we have a sequential caracterization of the $\gG$-convergence.
\begin{thm}\label{thmetric}
Let $X$ be a metric space. A sequence of functions $F_n$ $\gG$-converges to $F:X\to\overline{\R}$ if and only if the following two conditions hold:
\begin{itemize}
\item for every sequence $x_n$ converging to $x$, it holds $\displaystyle \varliminf_{n \to \infty} F_n(x_n)\ge F(x)$
\item for every $x\in X$ there exists a sequence $x_n$ converging to $x$ with $\displaystyle \varlimsup_{n\to \infty} F_n(x_n) \le F(x)$.
\end{itemize}
\end{thm}

\noindent Let now $W\in C^1(\R)$ be a double-well potential with minima in $\{0,1\}$, that is, 
$W(t)\ge 0$ for all $t\in \R$, and $W(t)= 0$ iff $t\in \{0,1\}$. We also assume 
$W(t)\ge C(t^2-1)$ for some $C>0$ and $t\in\R$. A typical example of such potential is
$W(t)=t^2(t-1)^2$.

\noindent For any $\eps>0$ we define the functionals $F_\eps:\Ldeu\to [0,+\infty]$ as
\[
F_\eps(u):= \left\{\begin{array}{ll}
\displaystyle \int_X  \left(\frac{\eps}{2}\normH{\nabla_H u}^2 + \frac{W(u)}{\eps} \right) d\gga 
& {\rm if\ }u\in H^1_\gga(X)
\\
\\
+\infty & {\rm if\ }u\in L^2_\gga(X)\setminus H^1_\gga(X)\,.
\end{array}\right.
\]

\noindent We are ready to prove our main $\gG$-convergence result.

\begin{thm}\label{modmortMal}
When $\eps$ tends to zero
the functionals $F_\eps$ $\gG$-converge, in the weak topology of $\Ldeu$, to the functional $c_W \barf$, where $c_W=\int_{0}^1 \sqrt{2W(t)}\, dt$.
\end{thm}

\begin{proof}
Notice first that the $\gG$-limit does not change if
we restrict the domain of $F_\eps$ to the functions $u\in H^1_\gga(X)$ such that $0\le u\le 1$.
This follows from the following two facts:
\begin{itemize}
\item[-] for all $u\in H^1_\gga(X)$, letting $\tilde u=\min(\max(u,0),1)$, we have $F_\eps(\tilde u)\le F_\eps(u)$;
\item[-] $F_\eps(u)\ge \int_X \frac{W(u)}{\eps}d\gga$ for all $u\in H^1_\gga(X)$, which implies that the $\gG$-limit
is concentrated on the functions $u\in\Ldeu$ such that $u(x)\in \{0,1\}$ for a.e. $x\in X$.
\end{itemize}
Since the restricted domain is contained in the unit ball of $\Ldeu$,
which is metrizable for the weak $\Ldeu$-topology, by Theorem \ref{thmetric} 
the $\gG$-limit and the sequential $\gG$-limit of $F_\eps$ coincide.

\smallskip

\noindent We now compute the $\gG$-liminf of $F_\eps$. 

\noindent Let $u_\eps\in H^1_\gga(X)$ be such that $0\le u_\eps\le 1$
and $F_\eps(u_\eps)\le C$ for some $C>0$.
Then $\int_X W(u_\eps)d\gga \le C\eps$, which gives a uniform bound on $\|u_\eps\|_{\Ldeu}$
recalling that $W(u)\ge C(u^2-1)$. 
As a consequence, there exists a weakly converging subsequence, still denoted by $u_\eps$. 
Letting $u$ be its weak limit, from $0\le u_\eps\le 1$ we get $0\le u\le 1$. 
Using the coarea formula \eqref{coarea}, we obtain the estimate
\begin{align*}
F_\eps(u_\eps) &= \int_X  \left(\frac{\eps}{2}\normH{\nabla_H u}^2+ \frac{W(u)}{\eps}\right) d\gga \\
&\ge \int_X\sqrt{2W(u_\eps)} \, \normH{\nabla_H u}\, d\gga\\
&=\int_{0}^{1} \sqrt{2W(t)} \, \Pgamma{\{u_\eps >t\}}\, dt\,.
\end{align*}

\noindent Fix now $\delta>0$. From the fact that $\vgamma{\{\delta\le u_\eps\le 1-\delta\}}\to 0$ as $\eps\to 0$,
it follows that, for every sequence $t_\eps \in [\delta,1-\delta]$,
then functions $\chi_{\{u_\eps >t_\eps\}}$ tend weakly to $u$ in $\Ldeu$. 
For every $\eps>0$ let us choose $t_\eps \in [\delta,1-\delta]$ such that 
\[
\int_{\delta}^{1-\delta} \sqrt{2W(t)}\Pgamma{\{u_\eps>t\}}dt \ge \left(\int_{\delta}^{1-\delta} \sqrt{2W(t)}dt \right)\Pgamma{\{u_\eps>t_\eps\}}.
\]
Then, by Theorem \ref{relax} we have
\begin{align*}
\varliminf_{\eps \to 0} F_\eps(u_\eps)&\ge \varliminf_{\eps\to 0} \left(\int_{\delta}^{1-\delta} \sqrt{2W(t)}dt \right)\Pgamma{\{u_\eps>t_\eps\}}\\
&\ge  \left(\int_{\delta}^{1-\delta} \sqrt{2W(t)}dt \right) \barf(u)\,.
\end{align*}
Since $\delta$ is arbitrary we get the $\gG$-liminf inequality.
 
\smallskip
 
\noindent The $\gG$-limsup is done similarly to the (Euclidean) finite dimensional case \cite{modicamortola,modica}. 
Since $\barf$ is the relaxation of $F$ in the weak $\Ldeu$-topology and since we can approximate sets 
of finite perimeter by smooth cylindrical sets by Proposition \ref{denscyl}, 
for every $u\in BV_\gga(X)$ with $0\le u\le 1$ there exists a sequence $E_n$ 
of smooth cylindrical sets with $\chi_{E_n}$ converging weakly to $u$ and such 
that $\Pgamma{E_n}$ tends to $\barf(u)$. This shows that we can restrict ourselves 
to smooth cylindrical sets for computing the $\gG$-limsup of $F_\eps$.
 
\noindent Let $m\in\mathbb N$ and $E=\Pi_m^{-1}(E_m)$, where $E_m\subset\R^m$ is a smooth set with finite Gaussian perimeter, and let 
\[ d^H(x,E):=d(\Pi_m(x),E_m)\]
where $d(x,E_m)$ is the usual distance function from $E_m$ in $\R^m$. Notice that 
\[
d^H(x,E)=\min \{ \normH{x-y} ; y\in E, x-y \in H\},
\]
moreover $d^H$ is differentiable almost everywhere 
with $\normH{\nabla_H d^H(x,E)}=1$. 

\noindent Let $\gd>0$, $\alpha_\gd:=\max\{W(t) \; : \; t\in [0,\delta] \cup [1-\delta,1]\}$ and define  
$W_\gd,\,H_\gd : [0,1]\rightarrow \R$ as
\begin{eqnarray*}
W_\gd(t)&:=&\left\{\begin{array}{ll}
\alpha_\gd & \textrm{if } 0\le t\le \delta\\
W(t) & \textrm{if } \gd\le t\le 1-\gd\\
\alpha_\gd & \textrm{if } 1-\gd\le t\le 1.\end{array}\right.
\\
\\
H_\gd(t)&:=&\int_0^t\frac{1}{\sqrt{2W_\gd(s)}}ds.
\end{eqnarray*} 
Finally let $\eta_\delta$ be the usual truncated one-dimensional transition profile defined as
\[\eta_\gd(t):=\left\{\begin{array}{ll}
0& \textrm{if } t\le 0\\
H_\gd^{-1}(t) &\textrm{if } 0\le t\le H_\gd(1)\\
1 & \textrm{if } t> H_\gd(1).\end{array}\right.\]

\noindent Observe that $\eta_\gd$ is a Lipschitz function which verifies $\frac{\eta_\delta'^2}{2}=W_\delta(\eta_\delta)$. We then set 
\[u_\eps(x):=\eta_\delta\left(\frac{d^H(x,E)}{\eps}\right).\]
 We finally have
 \begin{align*}
 F_\eps(u_\eps)&=\int_X \left(\frac{\eps}{2}\normH{\nabla_H u_\eps}^2  + \frac{W(u_\eps)}{\eps}\right) d\gga\\
 								&\le \int_X \left(\frac{\eps}{2}\normH{\nabla_H u_\eps}^2  + \frac{W_\delta(u_\eps)}{\eps}\right) d\gga\\
 								&=\int_{X} \frac{\eps}{2}{\eta_\delta'}^2\left(\frac{d(\Pi_m(x))}{\eps}\right)\left(\frac{|\nabla_{H} d(\Pi_m(x))|}{\eps}\right)^2\\
 								&\qquad + \frac{1}{\eps}W_\delta\left(\frac{\eta_\delta(d(\Pi_m(x)))}{\eps}\right) d\gga\\
 								&=\int_{\R^m} \left[\frac{1}{2}{\eta_\delta'}^2\left(\frac{d}{\eps}\right] 
 								+W_\delta\left(\eta_\delta\left(\frac{d}{\eps}\right)\right) \right)                           \frac{|\nabla d|}{\eps}\,d\gga_m\\
								&=\int_0^{H_\gd(1)} \left(\frac{{\eta_\delta'}^2(t)}{2}+ W_\delta(\eta_\delta(t))\right)P_{\gamma_m}(\{d>\eps t\})\,dt. 				
 								\end{align*}
 		The proof is completed since for every $t\in[0,H_\gd(1)]$, $P_{\gamma_m}(\{d>\eps t\})$ tends to $P_{\gamma_m}(E_m)$ as $\eps\to 0$, and 
 		\[\int_{0}^{H_\gd(1)} \left(\frac{\eta_\delta'^2(t)}{2}+ W_\delta(\eta_\delta(t))\right)\,dt=\int_{0}^1 \sqrt{2W_\delta(t)}\,dt\,.\]
 		Thus we have 
 		\[\varlimsup_{\eps\to 0} F_\eps(u_\eps)\le \left(\int_{0}^1 \sqrt{2W_\delta(t)}\,dt\right)P_{\gga_m}(E_m),\]
 		which gives the desired inequality letting $\delta\to 0$ and $m\to+\infty$.
\end{proof}

\begin{remarque}\rm As in the Euclidean case, a similar result can be proven for the volume constrained problems. 
In this case, the proof of the $\gG$-liminf is exactly the same as in Theorem \ref{modmortMal}, 
and the $\gG$-limsup is also very similar. The only difference comes 
from the fact that we have to adapt the recovery sequence to have the right volume, 
and this can be done as in \cite{modica} by slightly translating $\eta_\delta$.
\end{remarque}

\noindent We now state some simple implications of the $\gG$-convergence result.
\begin{prop}
Let $m\in[0,1]$ and $u_\eps$ be a minimizer of 
\begin{equation}\label{probvol}
\min_{\int_X u \,d\gga=m} \int_X  \left(\frac{\eps}{2}\normH{\nabla_H u}^2 + \frac{W(u)}{\eps} \right) d\gga 
\end{equation}
then $u_\eps=v_\eps(\hath_\eps(x))$ for some $\hath_\eps \in \cH$  with $\normH{h_\eps}=1$ and some $v_\eps$ minimizer of the one-dimensional problem
\begin{equation}\label{prob1d}
\min_{\int_{\R} v d\gga_1 =m} \int_\R  \frac{\eps}{2} v'^2 d\gga +\int_\R \frac{W(v)}{\eps} d\gga_1. 
\end{equation}
in particular, $v_\eps$ (strongly) converges to the characteristic function of a half-line.
\end{prop}

\begin{proof}
For every $u \in H^1_\gga(X)$, by Proposition \ref{ehrfunc}, we have $\int_X u^* d\gga=\int_X u d\gga$ and $F_\eps(u^*)\le F_\eps(u)$, with equality only if $u$ is of the form $u(x)=v(\hath(x))$ for some $\hath\in\cH$ with $\normH{h}=1$. Using that $\hath$ is the limit in $\Ldeu$ of linear functions of the form $R^*x_i^*$, it is readily seen that $\nabla_H \hath=h$, and thus we get
\[
F_\eps(u)=\int_X  \left(\frac{\eps}{2} v'(\hath(x))^2 + \frac{W(v(\hath(x)))}{\eps} \right)d\gga
=\int_\R \left( \frac{\eps}{2} v'^2 d\gga +\int_\R \frac{W(v)}{\eps}\right) d\gga_1.
\]
Therefore problem \eqref{probvol} reduces to the one-dimensional problem \eqref{prob1d}.

\noindent Using the compact embedding of $H^1_{\gga_1}(\R)$ in $L^2_{\gga_1}(\R)$ (see \cite[Th. 4.10]{AMMP}) 
and the direct method of the calculus of variations, we get that \eqref{prob1d} has a minimizer. 
Moreover, by the $\gG$-convergence of the one-dimensional functionals in the strong $L^2_{\gga_1}(\R)$-topology towards the a multiple of the perimeter 
(which can be obtained exactly as in the classical Modica-Mortola Theorem since compact embedding of $BV_{\gga_1}(\R)$ in $L^1_{\gga_1}(\R)$ holds), 
we find that every sequence of minimizers $v_\eps$ of \eqref{prob1d} has a subsequence strongly converging towards the 
characteristic of the half-line of measure $m$.   
\end{proof}

\noindent We finally give another convergence result for the prescribed curvature problem in case of uniqueness of minimizers.

\begin{prop}
Let $g\in \Ldeu$, then the following assertions are equivalent:
\begin{itemize}
\item the functional 
\begin{equation}\label{PPg}
F_g(E)=\Pgamma{E}+\int_E g d\gga\end{equation}
 has a unique minimizer in the class of sets of finite perimeter;
\item the functional 
\begin{equation}\label{Pug}
\barf_g(u)=\barf (u) + \int_X u g d\gga \end{equation}
has a unique minimizer in $ BV_\gga(X)$.
\end{itemize}
Moreover, when this holds the two minimizers  coincides. Finally, if $u_\eps$ is a sequence in $H^1_\gga(X)$ satisfying 
\begin{equation*}\sup_\eps \left(F_\eps(u_\eps) +\int_X u_\eps g d\gga \right) \le C\end{equation*}
for some $C>0$,
then $u_\eps$ has a subsequence strongly converging to $\chi_E$ in $\Ldeu$, where $E$ is the common minimizer of \eqref{PPg} and \eqref{Pug}.
\end{prop}

\begin{proof}
We first notice that the problem \eqref{PPg} always has a solution. Indeed, arguing as in \cite{CaMiNo}, if $E_n$ is a minimizing sequence 
for \eqref{PPg}, it has a  subsequence weakly converging to some $u\in BV_\gga(X)$. 
By the lower semicontinuity of the total variation and the coarea formula we then have
\[
\inf_E \left(\Pgamma{E}+\int_E g d\gga \right)\ge \totvar{u}+\int_X u g d\gga
=\int_{0}^{1} \left(\Pgamma{ \{u> t\}} +\int_{\{u>t\}} g(x) d\gga(x) \right) dt
\]
and thus the sets $\{u>t \}$ minimize $F_g$ for almost every $t$.
As $\barf$ is the relaxation of the perimeter we have that the minimum values in \eqref{PPg} and \eqref{Pug} are the same and thus any minimizer of $F_g$ is also a minimizer of $\barf_g$. This shows that if uniqueness does not hold in \eqref{PPg} then it does not hold in \eqref{Pug}, too. 
Now, if $u$ is a minimizer of $\barf_g$, applying the coarea formula once again we get
\[\inf_E F_g(E)= \barf_g (u) \ge \totvar{u} +\int_X ug d\gga =\int_{0}^1 \left(\Pgamma{ \{u> t\}} +\int_{\{u>t\}} g(x) d\gga(x) \right) dt.\]

\noindent As above, this implies that  $\{u>t \}$ solves \eqref{PPg} for almost every $t$. Therefore, if the minimizer of  $\barf_g$ is not a characteristic function, then uniqueness does not hold neither in \eqref{PPg} nor in \eqref{Pug}. This proves the first part of the Proposition.

\noindent The second statement easily follows from Theorem \ref{modmortMal}.  
Indeed, as the functionals $F_\eps(u) +\int_X ug d\gga$ $\gG$-converge to $\barf_g$ in the weak $\Ldeu$-topology, 
for every sequence $u_\eps$ bounded in energy, there exists a subsequence weakly converging to $\chi_E$ 
(where $E$ is the unique minimizer of \eqref{PPg} and \eqref{Pug}). However, by the lower semicontinuity of the norm,
\[m^{\frac{1}{2}}\ge \varliminf_{\eps\to 0}  \|u_\eps\|_{\Ldeu}\ge \|\chi_E\|_{\Ldeu}=m^{\frac{1}{2}}.\]
Thus $\|u_\eps\|_{\Ldeu}$ converges to $\|\chi_E\|_{\Ldeu}$, which implies the strong convergence of $u_\eps$.
\end{proof}

\begin{remarque}\rm
In \cite{GN}, we provide an example of functionals for which uniqueness of minimizers holds, namely
\[
\Pgamma{E}+\int_X (g-\gl)\, d\gga
\]
where $g:X\to \R$ is convex and $\gl\in (0,+\infty)$ is large enough.
\end{remarque}


\begin{thebibliography}{}

\end{thebibliography}


\begin{thebibliography}{100}
\bibitem{ADPP} {\sc L. Ambrosio, G. Da Prato and D. Pallara}, {\em $BV$ functions in a Hilbert space with
respect to a Gaussian measure}, Rend. Acc. Lincei, to appear.
\bibitem{AFP} {\sc L.~Ambrosio, N.~Fusco and D.~Pallara}, {\em Functions of bounded variation and free discontinuity problems}, Oxford Science Publications, 2000.
\bibitem{AF} {\sc L.~Ambrosio and A.~Figalli}, {\em Surface measures and convergence of the Ornstein-Uhlenbeck semigroup in Wiener spaces}, Ann. Fac. Sci. Toul. Math., to appear.
\bibitem{AMMP} {\sc L.~Ambrosio, S.~Maniglia, M.~Miranda Jr. and D.~Pallara}, {\em BV functions in abstract Wiener spaces } 
J. Funct. Anal. 258(3):785-813, 2010.
\bibitem{AMP} {\sc L.~Ambrosio, M.~Miranda Jr. and D.~Pallara}, {\em Sets with finite perimeter in Wiener
spaces, perimeter measure and boundary rectifiability}, Discr. Cont. Dyn. Syst. A 28:591-606, 2010.
\bibitem{bakryledoux} {\sc D.~Bakry and M.~Ledoux}, 
{\em L\'evy-Gromov's isoperimetric inequality for an infinite dimensional diffusion generator}, 
Invent. Math. 123:259-281, 1996.
\bibitem{bobkov} {\sc S.G.~Bobkov}, 
{\em An isoperimetric inequality on the discrete cube, and an elementary proof of the isoperimetric inequality in Gauss space}, 
Ann. Probab. 25(1):206-214, 1997. 
\bibitem{boga} {\sc V.I.~Bogachev}, {\em Gaussian measures}, Mathematical Surveys and Monographs 62, AMS, Providence, 1998.
\bibitem{BZ} {\sc J.~Brothers and W~.Ziemer}, {\em Minimal rearrangements of Sobolev functions}, 
J. Reine. Angew. Math. 384:153-179, 1988.
\bibitem{CK} {\sc E.~Carlen and C.~Kerce}, {\em On the cases of equality in Bobkov's inequality
and Gaussian rearrangement}, Calc. Var. PDE 13:1-18, 2001.
\bibitem{CLMN} {\sc V.~Caselles, A.~Lunardi, M.~Miranda Jr. and M.~Novaga}, 
{\em Perimeter of sublevel sets in infinite dimensional spaces}, Adv. Calc. Var, to appear.
\bibitem{CaMiNo} {\sc V.~Caselles, M.~Miranda Jr. and M.~Novaga}, 
{\em Total Variation and Cheeger sets in Gauss space}, J. Funct. Anal. 259(6):1491-1516, 2010.
\bibitem{CEFT} {\sc A.~Cianchi, L.~Esposito, N~Fusco and C.~Trombetti}, 
{\em A quantitative P\'olya-Szeg\"o principle}, J. Reine Angew. Math. 614:153-189, 2008.
\bibitem{CF} {\sc A.~Cianchi and  N~.Fusco}, {\em Functions of bounded variation and rearrangements}, Arch. Rat. Mech. Anal. 165:1-40, 2002.
\bibitem{CFMP} {\sc A.~Cianchi, N.~Fusco, F.~Maggi and A.~Pratelli}, 
{\em On the isoperimetric deficit in the Gauss space}, Amer. J. of Math. 133(1):131-186, 2011.
\bibitem{DM} {\sc G.~Dal Maso}, {\em An introduction to $\gG$-convergence}, Birkh\"auser, Boston, 1993. 
\bibitem{DemTem} {\sc F.~Demengel and R.~Temam}, {\em Convex functions of a measure and applications}, Indiana Univ. Math. J. 33:673-709, 1984.
\bibitem{ehr2} {\sc A.~Ehrhard}, {\em Sym\'etrisation dans l'espace de Gauss}, Math. Scand. 53(2):281-301, 1983.
\bibitem{ehr} {\sc A.~Ehrhard}, {\em In\'egalit\'es isop\'erim\'etriques et int\'egrales de Dirichlet gaussiennes}, 
Ann. Sci. Ecole Norm. Sup. 17:317-332, 1984.
\bibitem{EG} {\sc L. C.~Evans and G.~Gariepy}, {\em Measure Theory and Fine Properties of Functions}, Studies in Advanced Mathematics,
CRC Press, 1992.
\bibitem{fuku} {\sc M.~Fukushima}, {\em BV functions and distorted Ornstein-Uhlenbeck processes over the
abstract Wiener space}, J. Funct. Anal., 174:227-249, 2000.
\bibitem{fukuhino} {\sc M.~Fukushima and M.~Hino},  
{\em On the space of BV functions and a related stochastic calculus in infinite dimensions}, J. Funct. Anal. 183:245-268, 2001. 
\bibitem{giusti} {\sc E.~Giusti}, {\em Minimal surfaces and functions of bounded variation}, Monographs in Mathematics 80, Birkh\"auser, 1984.
\bibitem{GN} {\sc M.~Goldman and M.~Novaga}, {\em Convex minimizers for infinite dimensional variational problems}, Preprint, 2011.
\bibitem{liebloss} {\sc E.~Lieb and M.~Loss}, {\em Analysis}, Graduate Studies in Mathematics 14, AMS, Providence, 2001.
\bibitem{ledoux} {\sc M.~Ledoux}, {\em Isoperimetry and Gaussian Analysis}, Lecture Notes in Math. 1648, Springer, pp.165-294, 1996.
\bibitem{malliavin} {\sc P.~Malliavin}, {\em Stochastic analysis}, Grundlehren der Mathematischen, Springer, 1997.
\bibitem{modica} {\sc L.~Modica}, {\em The gradient theory of phase transitions and the minimal interface criterion}, 
Arch. Rat. Mech. Anal. 98:123-142, 1987.
\bibitem{modicamortola} {\sc L.~Modica and S.~Mortola}, {\em Un esempio di $\Gamma^-$-convergenza}, Boll. Unione Mat. Ital. 14(5):285-299, 1977.
\end{thebibliography}
\end{document}